\documentclass[11pt]{article}
\usepackage{amsthm}
\usepackage{amsmath,amssymb,amscd,eucal}
 \usepackage{latexsym}
\usepackage{graphicx}
\usepackage{tikz}
\usepackage{color} 

\usepackage{amsmath,amssymb,amscd,eucal} 
\usepackage{times,tikz}
\usetikzlibrary{shapes.geometric,shapes.arrows,decorations.pathmorphing}
\usetikzlibrary{matrix,chains,scopes,positioning,arrows,fit}
\usepackage{amssymb}
\usepackage{latexsym}

\newtheorem{thm}[equation]{Theorem}

\newtheorem{prop}[equation]{Proposition} 
 
\newtheorem{remark}[equation]{Remark}

\numberwithin{equation}{section}
 
\newcommand{\angstrom}{\mbox{\normalfont\AA}} 

\renewcommand{\det}{\mathsf{det}}
\renewcommand{\max}{\mathsf{max}}
\newcommand{\overbar}[1]{\mkern 1.2mu\overline{\mkern-1mu#1\mkern-1mu}\mkern 1mu}

\newcommand{\CC}{\mathbb{C}}   

\newcommand{\Cs}{\mathbf{C}} 
   
\newcommand{\DD}{\mathbf{D}}
\newcommand{\dr}{\mathrm{d}} 
 
\newcommand{\arm}{\mathrm{a}}

\newcommand{\vs}{\mathsf{v}}

\newcommand{\GG}{\mathsf{G}} 
\newcommand{\II}{\mathbf{I}} 
\newcommand{\lam}{\lambda}  
\newcommand{\mm}{\mathsf{m}} 
\newcommand{\Ms}{\mathsf{M}} 
\newcommand{\OO}{\mathbf{O}}
\newcommand{\SU}{\mathsf{SU}} 
 
\newcommand{\pr}{\mathrm{p}}

\newcommand{\Sr}{\mathbf{S}} 
\newcommand{\Ss}{\mathsf{S}} 
\newcommand{\Tr}{\mathrm{T}}   
\newcommand{\Ts}{\mathsf{T}} 
\newcommand{\TT}{\mathbf{T}}
\newcommand{\Ur}{\mathrm{U}}
\newcommand{\Vs}{\mathsf{V}}  
\newcommand{\what}{\widehat} 
\newcommand{\Zs}{\mathsf{Z}}

\newcommand\End{\mathsf {End}}

\newcommand\half{\frac{1}{2}}
\def \ot {\otimes} 
\def\modd{\, \mathsf{mod} \,}

\def\dimm{\, \mathsf{dim}\,}

\begin{document} 
\title{Poincar\'e Series for Tensor Invariants and the McKay Correspondence}
\author{Georgia Benkart} 
 \date{}
\maketitle

\begin{abstract} For a finite group $\GG$ and a finite-dimensional $\GG$-module $\Vs$,
we prove a general result on the Poincar\'e series for the $\GG$-invariants in the tensor
algebra $\Ts(\Vs) = \bigoplus_{k \geq 0} \Vs^{\ot k}$.  We apply this result to the finite
subgroups $\GG$ of the $2 \times 2$ special unitary  matrices  and their natural 
module $\Vs$ of $2 \times 1$ column vectors.  Because these subgroups are in one-to-one correspondence
with the simply laced affine Dynkin diagrams by the McKay correspondence, the Poincar\'e series obtained are the generating functions
for  the number of walks on the simply laced affine Dynkin diagrams.   
\end{abstract}
\textbf{MSC Numbers (2010)}: 14E16, 05E10, 20C05   \hfill \newline
\textbf{Keywords}:   
tensor invariants, Poincar\'e series, McKay correspondence, Schur-Weyl duality 

\maketitle

\maketitle 
\begin{section} {Introduction} \end{section} 
Let $\GG$ be a group,  and assume $\{\GG^\lambda \mid \lambda \in \Lambda(\GG)\}$ is
the set of finite-dimensional irreducible $\GG$-modules over the complex field $\mathbb C$.   
Associated to a fixed finite-dimensional $\GG$-module  $\Vs$  over $\mathbb C$  
is the {\it representation graph} $\mathcal R_{\Vs}(\GG)$ having 
nodes  indexed by the  $\lambda$ in $ \Lambda(\GG)$  and  
  $a_{\mu,\lambda}$ edges from $\mu$ to $\lambda$ in
$\mathcal R_{\Vs}(\GG)$ if
\begin{equation}\label{eq:Got} \GG^\mu \ot \Vs  = \bigoplus_{\lambda \in \Lambda(\GG)}  a_{\mu, \lambda} \GG^\lambda.\end{equation} 
Thus, the number of edges $a_{\mu,\lambda}$ from $\mu$ to $\lambda$ in $\mathcal{R}_\Vs(\GG)$  is the multiplicity of
$\GG^\lambda$ as a summand of $\GG^\mu \ot \Vs$.    

When $\GG$ is a finite group, the representation graph and the characters of $\GG$ are closely related.   Assume  $\chi_\Vs$  is  the character of   $\Vs$,   and $\chi_\lambda$ is the character of  $\GG^\lambda$ for $\lambda \in \Lambda(\GG)$.  Let  $\mathsf{d} = \dimm \Vs =\chi_\Vs(1)$.   Steinberg \cite{Ste} has shown that when
the action of $\GG$ on $\Vs$ is faithful, the following hold:
\medskip

\begin{itemize}
\item The eigenvalues  of the matrix $ \big( \mathsf{d}\,\delta_{\mu,\lambda} - a_{\mu, \lambda}\big)$ are $\mathsf{d}-\chi_\Vs(g)$ as $g$ ranges  over a set $\Gamma$ of conjugacy class representatives of $\GG$.
\item The column vector
  $\big(\chi_\lambda(g)\big)$ with entries given by the
  character values of the irreducible $\GG$-modules at $g$ is an eigenvector corresponding to  $\mathsf{d}-\chi_\Vs(g)$.  These vectors form the columns of the character table of $\GG$.   
\item The vector $\big(\mathsf{d}^\lambda \big)$,  whose entries are the  dimensions $\mathsf{d}^\lambda = \dimm \GG^\lambda =\chi_\lambda(1)$ of the irreducible $\GG$-modules,   corresponds to the eigenvalue 0. 
\end{itemize} 

Let  $\GG^{0}$ be  the one-dimensional trivial $\GG$-module  on which every element of the group  acts as the identity transformation, and let   $\mm_{k}^\lam$  be the number of walks of $k$ steps
from $0$ to $\lam$ on the representation graph $\mathcal{R}_\Vs(\GG)$.    Since each step on the graph is accomplished by  tensoring with $\Vs$,   
$\mm_{k}^\lambda$ is the multiplicity of the irreducible $\GG$-module  $\GG^\lambda$ in
$\GG^0 \otimes \Vs^{\ot k} \cong \Vs^{\ot k}$. 

In what follows, we identify   $\Vs^{\ot 0} \cong \CC$ as a $\GG$-module  with  $\GG^{0}$,
so that $\mm_0^{\lam} = \delta_{\lam,0}$  (the Kronecker delta).      
For $\lam \in \Lambda(\GG)$,  we  consider the Poincar\'e series 
\begin{equation}\label{eq:ps} \mm^\lam(t) = \sum_{k \geq 0}  \mm_{k}^\lam \, t^k \end{equation}
for the multiplicity of $\GG^\lam$ in  the tensor algebra $\Ts(\Vs) =  \bigoplus_{k \geq 0} \Vs^{\ot k}$,
(which is also the generating function  for the number of walks 
from 0 to $\lam$ in $\mathcal R_{\Vs}(\GG)$).   In particular,  $\mm^{0}(t)$  is the Poincar\'e
series for the $\GG$-invariants,   $\Ts(\Vs)^\GG =\{ w \in \Ts(\Vs)  \mid g w = w \ \hbox{for all} \ g \in \GG\}$,  in  $\Ts(\Vs)$.    

The \emph{centralizer algebra}, 
$$\Zs_k(\GG) = \{ X \in \End(\Vs^{\ot k}) \mid  X g w = g Xw \ \ \hbox{\rm for all} \ \
w \in \Vs^{\ot k}\},$$  plays an essential role in understanding the $\GG$-module
$\Vs^{\ot k}$.  The idempotents that project $\Vs^{\ot k}$ onto its irreducible
$\GG$-summands live in  the finite-dimensional semisimple associative algebra $\Zs_k(\GG)$.   
\emph{Schur-Weyl duality}  relates the decomposition of \ $\Vs^{\ot k}$ as a $\GG$-module to the decomposition of \ $\Vs^{\ot k}$  as a $\Zs_k(\GG)$-module revealing the following connections  between the representation theories of $\GG$ and
$\Zs_k(\GG)$: 
\begin{itemize}
\item \  The irreducible $\Zs_k(\GG)$-modules $\Zs_k^\lam$  are in bijection with
the elements $\lam$ of  $\Lambda_k(\GG) := \{\mu \in \Lambda(\GG) \mid 
 \mm_k^\mu \geq 1\}$. 
\item  \  \  $\Vs^{\ot k}  \,\cong \,  \displaystyle{\bigoplus_{\lambda \in \Lambda_k(\GG)}  \mm_k^\lambda \, \,\GG^\lambda  \quad \hbox{\rm and}   \quad  
\Vs^{\ot k}  \, \cong \,  \bigoplus_{\lambda \in \Lambda_k(\GG)} \mathsf {d}^\lambda \,\, \Zs_k^\lambda}$.
\item \   $\dimm \Zs_k^\lambda =  \mm_k^\lambda =$ \  number of walks of \ $k$ \ steps from \   0 \ to  $ \lambda$
on ${\mathcal R}_\Vs(\GG)$.
\item \  \   $\dimm \GG^\lambda =  \mathsf{d}^\lambda$.
 \item  
 \  \  $\dimm\, \Zs_k(\GG) = \displaystyle{\sum_{\lambda \in \Lambda_k(\GG)} (\dimm \, \Zs_k^\lambda)^2 =  \sum_{\lambda
  \in \Lambda_k(\GG)} (\mm_k^\lambda)^2}$.
  
 \item  \ \  When the dual module $\Vs^*$ is isomorphic to $\Vs$ as a $\GG$-module, then
 
\  $\dimm\, \Zs_k(\GG) =$ \ number  of walks of \ $2k$ \ steps from \  $0$ \  to \ $0$ \  on \  ${\mathcal R}_\Vs(\GG)$ 
  
 \qquad \qquad \quad  \   $= \dimm \Zs_{2k}^0$.

   \item \  \  As a $(\GG, \Zs_k(\GG))$-bimodule,  $\Vs^{\ot k}$ has a multiplicity-free decomposition, 
\begin{center} {$\Vs^{\ot k}  \cong \displaystyle{\bigoplus_{\lambda \in \Lambda_k(\GG)}} \left(\GG^\lambda \otimes  \Zs_k^\lambda \right)$.} \end{center} 
\item   \  \ $\mm^\lam(t) = \sum_{k \geq 0} \mm_k^\lam \, t^k   =  \sum_{k \geq 0} \left( \dimm \Zs_k^\lam \right)\, t^k.$
  \end{itemize}  
  
Often one is able to build 
 the entire family  of finite-dimensional irreducible $\GG$-modules from a single well-chosen module $\Vs$ and its tensor powers by applying idempotents in the algebras $\Zs_k(\GG)$  to project onto the irreducible $\GG$-summands.
Schur's groundbreaking 1901 doctoral thesis constructed the finite-dimensional irreducible polynomial representations for the general linear group
  $\mathsf{GL}_n(\CC)$ from tensor powers of its defining module $\Vs = \mathbb C^n$ in exactly this way.  
The algebra $\Zs_k(\mathsf{GL}_n(\CC))$ is a homomorphic image of the group algebra $\CC \Sr_k$ of the
  symmetric group $\Sr_k$ for $k \geq 1$, which acts by permuting the factors of $\Vs^{\ot k}$.  
    
Our aim here is to establish a general result about
the Poincar\'e series  $\mm^\lam(t)$ for arbitrary finite groups  and then to apply this result to the finite subgroups
$\GG$ of the special unitary group 
$${\mathsf{SU}_2} = \displaystyle{ \left\{\begin{pmatrix} \ x & y \\ - \bar y  & \bar x \end{pmatrix} \, \bigg\vert\  x, y \in \CC, \ \,
x \bar x + y \bar y = 1 \right\}},$$ where  ``$\overbar{\ \ }$'' denotes complex conjugate.
These subgroups are especially interesting  because of the McKay correspondence and the
vast  literature it has inspired during the last 35 years. 

In 1980, J. McKay \cite{Mc1, Mc2}  discovered  a remarkable one-to-one correspondence between the isomorphism classes of the finite subgroups of  $\SU_2$  and the  simply laced affine Dynkin diagrams.   Almost a century earlier,  F. Klein had determined that a finite subgroup of $\SU_2$ must be isomorphic to one of the following:  (a) a cyclic group $\Cs_n$ of order $n$, (b) a binary dihedral group $\DD_n$ of order $4n$, or (c) one of the 3 exceptional groups: the binary tetrahedral group $\TT$ of order 24,  the binary octahedral group $\OO$ of order 48, or the binary icosahedral group $\II$ of order 120.   Binary here refers to the fact  that the  center is  $\{\pm \mathrm{I} \}$, where $ \mathrm{I}$ is
the $2 \times 2$ identity matrix, and the group modulo its center  is a dihedral group or the rotational symmetry group of a tetrahedron, octahedron, or icosahedron in the exceptional cases.   

The natural module for ${\mathsf{SU}_2}$ is the space 
$\Vs  =  \CC^2 =  \left \{\big(\begin{smallmatrix} \cdot \\ \cdot \end{smallmatrix} \big ) \right\}$ 
of $2 \times 1$ column vectors,  which ${\mathsf{SU}_2}$ and its finite
subgroups $\GG$  act on  by matrix multiplication.  McKay's observation was that the representation graph
$\mathcal R_{\Vs}(\GG)$ for  $\GG = \Cs_n, \DD_n, \TT, \OO, \II$ is exactly the affine Dynkin diagram $\mathsf{\what{A}_{n-1}}$, $\mathsf{\what{D}_{n+2}}$, $\mathsf{\what{E}_6}$, $\mathsf{\what{E}_7}$, $\mathsf{\what{E}_8}$, respectively,  with the
vertex 0 being the affine node.   Thus, the correspondence gives the pairings below.  The label inside the node is the index of the irreducible $\GG$-module, and the label above the node is its dimension, which is the mark on the Dynkin diagram.  The trivial module is indicated in white and the module $\Vs = \CC^2$ in black.  In the cyclic case $\Vs = \Cs_n^{(n-1)} \oplus \Cs_n^{(1)}$, 
and in all other cases $\Vs = \GG^{(1)}$.
\begin{equation}\label{eq:dd} 
\begin{array}{c} \hspace{-.7cm}  (\Cs_n, \what{\mathsf A}_{n-1}) \hspace{4.8cm}  (\DD_n, \what{\mathsf D}_{n+2}) \\\includegraphics[scale=.65,page=3]{poincare-diagrams.pdf} 
\end{array}  
\end{equation} 
$$
\begin{array}{c}\hspace{-1cm}  (\TT, \what{\mathsf E}_{6}) \hspace{5cm}  (\OO, \what{\mathsf E}_{7})\\ \includegraphics[scale=.6,page=4]{poincare-diagrams.pdf} \\
(\II, \what{\mathsf E}_{8})\\
\includegraphics[scale=.6,page=5]{poincare-diagrams.pdf} 
\end{array}
$$ 

Furthermore,  the following hold:
\begin{itemize}  
\item[{\rm (i)}] The sum   $h := \sum_{\lam \in \Lambda(\GG)} \dimm \GG^\lam$ of the dimensions (marks)  is  the Coxeter number   of the corresponding finite Dynkin diagram obtained by deleting the affine node.
\item[{\rm (ii)}] The marks on the nodes of the finite diagram, that is,  the dimensions of the nontrivial $\GG$-modules,   are the coefficients when the highest root of the corresponding root system is  expressed in terms of the simple roots   (see \cite{Bo} or  \cite{Ka} for more details).
\item[{\rm (iii)}] The Cartan matrix of the affine Dynkin diagram is $\mathrm{C} = 2\mathrm{I }- \mathrm{A}$, where $\mathrm{A}= (a_{\mu,\lambda})$ is the adjacency matrix of the graph $\mathcal R_{\Vs}(\GG)$ (i.e. the affine Dynkin diagram) and $\mathrm{I}$ is the identity matrix of the appropriate size.
\item[{\rm (iv)}] The marks are the coordinates of the Perron-Frobenius eigenvector of $\mathrm{A}$.
\item[{\rm (v)}] The eigenvectors  of $\mathrm{C}$ form  the character table of $\GG$. 
\end{itemize}
(Part (v) inspired Steinberg's results mentioned earlier.) 

In \cite{BBH} and \cite{BH} (see also \cite{Be}), we investigate the structure
and representations of  the centralizer algebras $\Zs_k(\GG)$
for the finite subgroups $\GG$ of $\SU_2$.  Among the results established in those papers
is a fruitful relationship between partition algebras and the centralizer algebras $\Zs_k(\GG)$ for $\GG = \TT$ and $\OO$.
Partition algebras were introduced by Martin \cite{M1}  to study the Potts lattice model 
of interacting spins in statistical mechanics, and they have been widely
studied in the last 25 years because of their connections with representations of symmetric groups (see  \cite{J2}, \cite{M2}, \cite{HR},
\cite{BDO}). 

It is well known that $\SU_2$  
has infinitely many finite-dimensional irreducible modules, 
$\Vs(k)$, \  $k=0,1,\dots, $   indexed by the nonnegative integers,
and  $\mathsf{dim}\, \Vs(k) = k+1$.  The module $\Vs(1)$  is
the natural $2$-dimensional $\SU_2$-module $\Vs$.  The module $\Vs(k)$ is isomorphic as an $\SU_2$-module
to the symmetric power $\Ss^k(\Vs)$ of $\Vs$, and hence, it can be identified with the space of
homogeneous polynomials of degree $k$ in two variables.    
The restriction of $\Vs(k)$ to a finite subgroup  $\GG$ of $\SU_2$ has
been investigated by many authors  (\cite{Sl1}, \cite{Sl2},  \cite{G-SV}, \cite{Kn}, \cite{Kos2}, \cite{Ste}) because of 
connections with  Kleinian
singularities.  The Poincar\'e series  $\mathsf{s}^\lam(t)$  for the multiplicities $\mathsf{s}_k^\lam$
of $\GG^\lam$ in the $\GG$-modules $\Ss^k(\Vs)$ for $k \geq 0$  
 has been shown to 
have a particularly  beautiful expression, 

\begin{equation}\label{eq:smu} \mathsf{s}^\lambda(t)   =   \sum_{k \geq 0} \mathsf{ s}_k^\lambda\, t^k =  \frac{ \mathsf{z}^\lam(t)}{(1-t^a)(1-t^b)},\end{equation}
where the numerator $\mathsf{z}^\lam(t)$ is a polynomial in $t$, 
\begin{align}\label{eq:ab}\begin{split}  a & = 2 \cdot \max\{\dimm \GG^\lam  \mid \lam \in \Lambda(\GG)\},  \ \ \ 
\hbox{\rm and}  \\
b &= h+2-a,  \ \ \ \hbox{\rm where $h$ is the Coxeter number.} \end{split} \end{align}  
Kostant (\cite{Kos1,Kos2}) (see also \cite{Kos3, Sp1, Sp2, Ste, Su}) has obtained exact formulas for the polynomials $\mathsf{z}^\lam(t)$
using orbits of an affine Coxeter element on the root system associated to $\GG$.    

The case $\lambda = 0$, which gives the Poincar\'e polynomial
for the $\GG$-invariants in $\Ss(\Vs) = \bigoplus_{k\geq 0} \Ss^k(\Vs)$,  has an especially
simple form (for extensions of this result to multiply laced Dynkin diagrams see \cite{Su} and \cite{Stk}).

\begin{prop}\label{P:syminv} {\rm (\cite{G-SV, Kn, Kos2})}  Let $\GG$ be a finite subgroup of $\SU_2$.    The Poincar\'e series for 
the $\GG$-invariants $\Ss(\Vs)^\GG$ in $\Ss(\Vs) = \bigoplus_{k \geq 0} \Ss^k(\Vs)$  is 
$$\mathsf{s}^0(t)   =  \frac{(1+t^h)}{(1-t^a)(1-t^b)},$$
where $a,b,h$ are as in \eqref{eq:ab}.   \end{prop} 
Using different methods, Springer \cite{Sp1} reproved Kostant's results on the Poincar\'e series
$\mathsf{s}^\lam(t)$ and in \cite{Sp2} used a generalization of Molien's formula,
\begin{equation}\label{eq:smol} \mathsf{s}^\lambda(t)   = \frac{1}{\vert \GG \vert}  \sum_{g \in \GG} \frac{\chi_\lambda(g)}{\det_\Vs(\mathrm{I} - g t)}, 
\end{equation}  to describe the character values $\chi_\lambda(g)$  for the exceptional polyhedral groups $\GG$.  
 In \cite{Ro}, Rossmann showed that the character values can also be obtained by evaluating the polynomials
$\mathsf{z}^\lam(t)$ in \eqref{eq:smu} at conjugacy class representatives of $\GG$.   
Suter \cite{Su}  adopted a
different though related way of studying the series $\mathsf{s}^\lambda(t)$ using quantum affine Cartan matrices. 
An extension of the results on the series $\mathsf{s}^\lambda(t)$  to ``semi-affine'' Dynkin diagrams has been given by McKay \cite{Mc3}. 

Our approach to determining the Poincar\'e series $\mm^\lambda(t)$ for the multiplicity
of $\GG^\lam$ in $\Ts(\Vs) = \bigoplus_{k \geq 0} \Vs^{\ot k}$  was inspired by
the methods in \cite{E}, \cite{Ro}, \cite{Su} and \cite{D}.   We first establish a  result 
for arbitrary finite groups and then apply it to finite subgroups of $\SU_2$.  
We show how this leads to new insights and results on the centralizer algebras
$\Zs_k(\GG)$ and on walks on the representation graph $\mathcal R_\Vs(\GG)$.  

\medskip
\textbf{Acknowledgments.} \  I am grateful 
to Dennis Stanton for answering my questions about Chebyshev polynomials.

\begin{section}{Poincar\'e series and Bratteli diagrams } \end{section} 

\begin{subsection}{Poincar\'e series for tensor multiplicities} \end{subsection}

For arbitrary finite groups, we prove  the following result on the Poincar\'e series for
tensor multiplicities. \smallskip 
 
\begin{thm}\label{T:main} Let $\GG$ be a finite group with irreducible modules $\GG^\lambda$, $\lambda \in \Lambda(\GG)$,  over $\CC$, and let $\Vs$ be a fixed finite-dimensional $\GG$-module such  that the action of $\GG$ on $\Vs$ is faithful, and the dual module $\Vs^*$  is isomorphic to $\Vs$ as a $\GG$-module.  Assume $\mm^\mu(t) = \sum_{k \geq 0} \mm_k^\mu t^k$ is the Poincar\'e series for the multiplicities $\mm_k^\mu$ ($k \geq 0$)  of
$\GG^\mu$ in  $\Ts(\Vs) = \bigoplus_{k \geq 0} \Vs^{\ot k}$. 
Let $\mathrm{A} = \big(a_{\mu,\lambda}\big)$ be the adjacency matrix of the 
representation graph $\mathcal R_{\Vs}(\GG)$,  and let  $\mathrm{M}^\mu$ be the 
matrix $\mathrm{I}-t\mathrm{A}$ with the column indexed by $\mu$ replaced by 
$\underline{\delta} = \left (\begin{smallmatrix}  1\\ 0 \\ \vdots \\  \\ 0 \end{smallmatrix}\right )$.   
Then 
\begin{equation}\label{eq:main} \mm^\mu(t) =   \frac{\det(\mathrm{M}^\mu)}{\det(\mathrm{I}-t\mathrm{A})} =
 \frac{\det(\mathrm{M}^\mu)}{\prod_{g \in \Gamma} \left(1- \chi_\Vs(g)t\right)}, \end{equation} 
where $\Gamma$ is a set of conjugacy class representatives of $\GG$. 
\end{thm} 

\begin{proof}   Our  proof of the first equality comes from the following computation, which uses the fact that  
the  multiplicity $\mm_k^\mu = \mathsf{dim}\left( \mathsf{Hom}_\GG(\GG^\mu, \Vs^{\ot k})\right)$.   
\begin{align}\label{eq:mt}  
\mm^\mu(t) &= \sum_{k \geq 0}  \mm_{k}^\mu \, t^k  =\sum_{k \geq 0}\mathsf{dim}\left( \mathsf{Hom}_\GG(\GG^\mu, \Vs^{\ot k})\right) t^k  \\
&= \delta_{\mu,0} + t \sum_{k \geq 1}\mathsf{dim}\left(\mathsf{Hom}_\GG(\GG^\mu, \Vs^{\ot k})\right)t^{k-1}  \nonumber \\
&= \delta_{\mu,0} + t \sum_{k \geq 1}\mathsf{dim}\left(\mathsf{Hom}_\GG(\GG^\mu \ot \Vs, \Vs^{\ot {(k-1)}})\right)t^{k-1} \qquad (\Vs \cong \Vs^*)\nonumber   \\ 
&= \delta_{\mu,0} + t \sum_{k \geq 0}\mathsf{dim}\left(  \mathsf{Hom}_\GG(\GG^\mu \ot \Vs, \Vs^{\ot {k}})\right)t^{k}\nonumber  \\
&= \delta_{\mu,0} +t \sum_{k \geq 0} \mathsf{dim}\left( \mathsf{Hom}_\GG\left( \sum_{\lambda \in \Lambda(\GG)} a_{\mu,\lambda} \GG^\lambda, \Vs^{\ot {k}}\right)\right)t^{k} \nonumber \\
&= \delta_{\mu,0} + t\sum_{\lambda \in \Lambda(\GG)} a_{\mu,\lambda}\sum_{k \geq 0}\mathsf{dim}\left(\mathsf{Hom}_\GG(\GG^\lambda, \Vs^{\ot k})\right)t^{k} \nonumber \\
&= \delta_{\mu,0} + t\sum_{\lambda \in \Lambda(\GG)} a_{\mu,\lambda}\left(\sum_{k \geq 0} \mm_k^\lambda \,  t^{k} \right) \nonumber\\
&= \delta_{\mu,0} +  t \sum_{\lambda \in \Lambda(\GG)} a_{\mu,\lambda} \  \mm^\lambda(t). \nonumber\end{align} 

Assuming  $\underline{\mm} = ( \mm^\lambda(t))$ is the column vector with entries $\mm^\lambda(t)$ as $\lambda$ ranges over
the elements of $\Lambda(\GG)$,  and  $\underline{\delta}$ is as in the theorem,
we have the following restatement of the result in \eqref{eq:mt} in matrix language,

$$(\mathrm{I} - t \mathrm{A}) \underline{\mm} = \underline{\delta}.$$
Now  $\mathrm{I} - t \mathrm{A}$ is an invertible matrix, since it is
equivalent to  the identity matrix $\mathrm{I}$  modulo the ideal of $\CC[t]$ generated by $t$.  The remainder of the proof of the first equality in \eqref{eq:main} just amounts to applying Cramer's rule
to solve for the series $\mm^\mu(t)$.      

Assume $\mathsf{d} = \dimm \Vs$,  and $\chi_\Vs$ is the character of $\Vs$.  Then by Steinberg's
result, the eigenvalues of $\mathsf{d}\mathrm{I}-\mathrm{A}$ are $\mathsf{d} - \chi_\Vs(g)$ 
where  $g \in\Gamma$, a set of conjugacy class representatives of $\GG$.  Hence,  
$$0 = \det\left(\big(\mathsf{d}-\chi_\Vs(g)\big)\mathrm{I} - \big(\mathsf{d}\mathrm{I}-\mathrm{A}\big)\right) =  \det (-\chi_\Vs(g)\mathrm{I} +\mathrm{A}), $$ which implies that 
$$\det(t \mathrm{I} - \mathrm{A})  =  \prod_{g \in \Gamma} \left(t-\chi_{\Vs}(g)\right).$$  
Replacing $t$ with $t^{-1}$ gives 
$\det(t^{-1} \mathrm{I} - \mathrm{A})  =  \prod_{g \in \Gamma} \left(t^{-1}-\chi_{\Vs}(g)\right)$, 
and multiplying both sides of that relation by $t^n$,  where  $n = \vert \Lambda(\GG)\vert =\vert \Gamma \vert$, then shows
$$ \det(\mathrm{I} - t\mathrm{A})  = t^n\det(t^{-1} \mathrm{I} - \mathrm{A})  = t^n \prod_{g \in \Gamma} \left(t^{-1}-\chi_{\Vs}(g)\right) = \prod_{g \in \Gamma}  \left(1-\chi_{\Vs}(g)t\right),$$
which provides the second equality in \eqref{eq:main}.  \end{proof} 

\begin{remark}\label{R:mmu}  When $\GG$ acts faithfully on $\Vs$, every irreducible $\GG$-module occurs
in some tensor power of $\Vs$, so  $\det(\Ms^\mu)$ and $\mm^\mu(t)$ are nonzero for all $\mu \in \Lambda(\GG)$.

 It is a consequence of \eqref{eq:mt}  that  
\begin{equation} \label{eq:mt2} \mm^\mu(t)  =  \delta_{\mu,0} +  t \sum_{\lambda \in \Lambda(\GG)} a_{\mu,\lambda} \  \mm^\lambda(t)
\end{equation}
which can be used to compute $\mm^\mu(t)$ from the series  $\mm^\lambda(t)$ for the nodes $\lambda$ connected to $\mu$ in the
representation graph $\mathcal R_{\Vs}(\GG)$.  
This is especially helpful (and efficient) for determining the Poincar\'e series for the finite subgroups $\GG$ of $\SU_2$. \end{remark} 

\begin{subsection} {$\Sr_4$ example}  \end{subsection} 
The irreducible modules for the symmetric group  $\Sr_n$ are indexed by partitions $\lam$  of $n$,  written
 $\lambda \vdash n$.  
Thus,  $\lambda$ is a  sequence $(\lambda_1,\dots, \lambda_\ell)$ of  weakly decreasing  nonnegative integers  such that the sum $\vert \lambda \vert:= \sum_{i=1}^\ell  \lambda_i = n$.  
In particular, when $n = 4$,  there are 5 irreducible modules $\Sr_4^\lambda$,  where $\lambda$ is one of the  following partitions:  $(4)$, $(3,1)$, $(2^2) = (2,2)$,
$(2,1^2) = (2,1,1)$, and $(1^4) = (1,1,1,1)$.    The module $\Sr_4^{(4)}$ indexed by the one-part 
partition (4) is the trivial $\Sr_4$-module, and $\Sr_4^{(1^4)}$ corresponds to the one-dimensional  sign representation.     
In the  character table for $\Sr_4$,  we have indicated a representative permutation 
for each conjugacy class across the top row.

\begin{table}[h]
$$\begin{tabular} {c|cccccccc} 
$\lambda\setminus g$ & $(1)$  & $(12)$  & $(123)$  & $(1234)$  & $(12)(34)$   \\ \hline
$(4)$  & $1$ & \ \, $1$ &\,\ $1$ & \ \, $1$ &\, \ $1$   \\
$(3,1)$  & $3$ &\, \ $1$ &\, \ $0$ & $-1$ & $-1$    \\
$(2^2)$ & $2$ &\, \ $0$ & $-1$ &\,  \ $0$ &\, \ $2$ \\
$(2,1^2)$ & $3$ & $-1$ &\, \  $0$ &\, \ $1$ & $-1$ \\
$(1^4)$ & $1$ & $-1$ &\, \ $1$ & $-1$ &\, \ $1$ 
\end{tabular}$$
\caption{Character table for $\Sr_4$} \label{table:char}
\end{table} 
Using the fact  that the character of a tensor product is the product of the characters of the factors,
we see that the representation graph $\mathcal R_{\Vs}(\Sr_4)$  for $\Vs = \Sr_4^{(3,1)}$ and its
adjacency matrix $\mathrm{A}$ are  \medskip

 \begin{figure}[ht!]
$\includegraphics[scale=1,page=1]{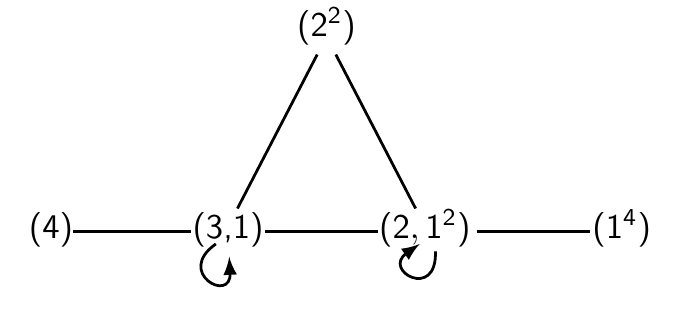}$ 

\vspace{-3cm} $\hspace{8cm} \mathrm{A} = \left(\begin{matrix} 0 & 1 & 0 & 0 & 0 \\ 1 & 1 & 1 & 1 & 0 \\
0 & 1 & 0 & 1 & 0 \\ 0 & 1 & 1 & 1 &  1\\ 0 & 0 & 0 & 1 & 0 \\ \end{matrix} \right)\  . 
$
\caption{Representation graph $\mathcal{R}_\Vs(\Sr_4)$ and its adjacency matrix $\mathrm A$  for $\Vs = \Sr_4^{(3,1)}$}\label{fig:Rs4}
\end{figure} 

Applying Theorem \ref{T:main}  and using  the second row of the character table, we have    
\begin{align*} \det(\mathrm{I} - t\mathrm{A}) & = \textstyle{\prod_{g \in \Gamma}} \big(1-\chi_{(3,1)}(g)\,t\big) = (1-3t)(1-t)(1+t)^2 \\
&= 1-2t-4t^2+2t^3 +3t^4. \end{align*}
This leads to the following expressions  for $\mm^\lam(t) =  \displaystyle{\frac{\det(\mathrm{M}^\lam)}{\det(\mathrm{I}-t\mathrm{A})}}$: 

\begin{align}\label{eq:S4ex}
\mm^{(4)}(t) = \displaystyle{\frac{1-2t-3t^2+t^3+t^4}{1-2t-4t^2+2t^3 +3t^4}} &=1+t^2+t^3+4t^4+10t^5+31t^6 + \cdots   \nonumber   \\   
\mm^{(3,1)}(t) = \displaystyle{\frac{t-t^2-2t^3}{1-2t-4t^2+2t^3 +3t^4}} & = t+t^2+4t^3+10t^4+31t^5+91t^6 + \cdots  \nonumber  \\
\mm^{(2^2)}(t) = \displaystyle{\frac{t^2-t^4}{1-2t-4t^2+2t^3 +3t^4}} & = t^2+2t^3+7t^4+20t^5+61t^6 + \cdots \nonumber  \\
\mm^{(2,1^2)}(t) = \displaystyle{\frac{t^2+t^3}{1-2t-4t^2+2t^3 +3t^4}} & = t^2+3t^3+10t^4+30t^5 +91 t^6 + \cdots \nonumber  \\
\mm^{(1^4)}(t) = \displaystyle{\frac{t^3+t^4}{1-2t-4t^2+2t^3 +3t^4}} & = t^3+3t^4+10t^5+30t^6 + \cdots \nonumber  \\
 \end{align}

Next, we connect the Poincar\'e series for tensor multiplicities in Theorem \ref{T:main}  to 
Bratteli diagrams. 
 
\begin{subsection}{Bratteli diagrams} \end{subsection} 
The \emph{Bratteli diagram} $\mathcal{B}_{\Vs}(\GG)$ associated to the group  $\GG$ and  the module $\Vs$ 
 is the infinite graph  with vertices labeled by the elements of  $\Lambda_k(\GG)$ on level $k$.      A walk of $k$ steps on the representation graph $\mathcal{R}_{\Vs}(\GG)$ from $0$ to $\lam$ is a sequence $\left(0,{\lam^1}, {\lam^2}, \ldots, {\lam^k}=\lam\right)$  starting at $\lam^0 = 0$,   such that  $\lam^j \in \Lambda(\GG)$   for each $1 \le j \le k$, and  $\lam^{j-1}$ is connected to  ${\lam^{j}}$ by an edge in $\mathcal{R}_{\Vs}(\GG)$.   Such a walk is equivalent to a unique path of length $k$ on the Bratteli diagram $\mathcal{B}_{\Vs}(\GG)$ from $0$
 at the top  to $\lam \in \Lambda_k(\GG)$ on level $k$.   
 
The subscript on vertex  $\lam \in \Lambda_k(\GG)$ in $\mathcal{B}_\Vs(\GG)$  indicates  the number  
$\mm_k^\lam$  of paths from $0$ on the top  to $\lam$ at level $k$ (hence, the number of walks on $\mathcal R_\Vs(\GG)$ of $k$ steps
from $0$ to $\lam$).  It can be easily  computed by summing, in a Pascal triangle fashion, the subscripts of the vertices at level $k-1$ that are connected to $\lam$.  This is  the multiplicity of
$\GG^\lam$ in $\Vs^{\ot k}$, which is also the 
dimension of the irreducible $\Zs_k(\GG)$-module  $\Zs_k^\lam$ by Schur-Weyl duality.  The sum of the squares of
those dimensions at level $k$ is the number on the right, which is the dimension of the centralizer algebra $ \Zs_k(\GG)$.
 
 The Bratteli diagram  for the $\mathbf{S}_4$ example in the previous section  is displayed in Figure \ref{fig:Bs4} below.
 The coefficients of the series in \eqref{eq:S4ex}  are the multiplicities $\mm_k^\lam$, 
and hence,  they are the subscripts on the node $\lam$ 
in the Bratteli diagram reading down the column containing $\lambda$.   They are the dimensions of the irreducible modules $\Zs_k^\lam$ for the centralizer algebras $\Zs_k(\Sr_4) = \End_{\Sr_4}(\Vs^{\ot k})$ for $\Vs = \Sr_4^{(3,1)}$ and $k=0,1,\dots$.  The sum of the
squares of the subscripts in a given row $k$ is the dimension of $\Zs_k(\Sr_4)$ and is the
number on the right.    The series $\mm^{(4)}(t)$ is the Poincar\'e series  for the $\Sr_4$-invariants
in $\Ts(\Vs) = \bigoplus_{k \geq 0} \Vs^{\ot k}$.  The series $ \sum_{k \geq 0} \mm_{2k}^{(4)} t^{2k}$, which corresponds to 
the right-hand column,  is the generating function 
for the dimensions $\dimm \Zs_k(\Sr_4)$ of the centralizer algebras (and also for the $\Sr_4$-invariants in $\Ts^{\mathsf{even}}(\Vs) = \bigoplus_{k \geq 0} \Vs^{\ot (2k)}$\,).    
\begin{figure}[ht!] 
$$
\begin{array}{c} \includegraphics[scale=1,page=2]{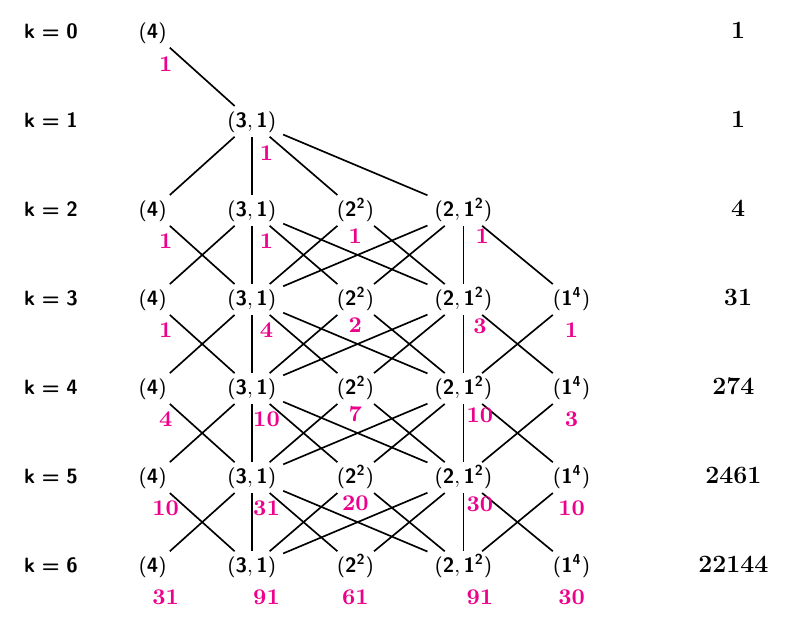} \end{array}
$$
\vspace{-.3cm} 
\caption{Levels $k=0,1,\dots,6$ of the Bratteli diagram  $\mathcal{B}_\Vs(\Sr_4)$ for $\Vs = \Sr_4^{(3,1)}$}\label{fig:Bs4}
\end{figure} 

\medskip
  
\begin{section}{ Poincar\'e series for the finite subgroups of $\SU_2$} \end{section} 
When  $\GG$ is one of the finite subgroups $\Cs_n$, $\DD_n$, $\TT$, $\OO$, $\II$ of $\SU_2$,
and $\Vs = \CC^2$,  the defining 2-dimensional module for $\GG$, we have the following immediate consequence of  the McKay correspondence and Theorem \ref{T:main}.
\medskip
 
\begin{thm} \label{T:mainsu}  Let $\GG$ be a finite subgroup of $\SU_2$ and $\Vs = \CC^2$.  Then the Poincar\'e series for the $\GG$-invariants $\Ts(\Vs)^\GG$  in $\Ts(\Vs) = \bigoplus_{k \geq 0}\Vs^{\ot k}$ is

\begin{equation}\label{eq:inv}\mm^{0}(t) =  \frac{\det \left(\mathrm{I} - t \angstrom \right)}{\det \left(\mathrm{I} - t \mathrm{A}\right)}
=  \frac{\det \left(\mathrm{I} - t  \angstrom \right)}{\prod_{g \in \Gamma} \left(1- \chi_\Vs(g)t\right)},
\end{equation}
where $\mathrm{A}$ is the adjacency matrix of the representation graph $\mathcal{R}_{\Vs}(\GG)$ (i.e. the affine Dynkin diagram corresponding to $\GG$ in \eqref{eq:dd});  \ $\angstrom$ is the adjacency matrix of the finite Dynkin diagram
obtained by removing the affine node;  and $\chi_\Vs(g)$ is the value of the character
$\chi_{\Vs}$ at $g \in \Gamma$, a set of conjugacy class representatives for $\GG$.
\end{thm}

\begin{remark}  Theorem \ref{T:mainsu}  can be regarded as an analog of Ebeling's theorem \cite{E} (see also \cite[Sec.~5.5]{Stk})
for  finite subgroups $\GG$ of $\SU_2$,  which 
relates the Poincar\'e series $\mathsf{s}^0(t)$ for the $\GG$-invariants in  the symmetric algebra $\mathsf{S}(\Vs) = \bigoplus_{k\geq 0} \mathsf{S}^k(\Vs)$  to the characteristic polynomial $\mathsf{ch}^{{}^{\boldsymbol{\circ}}}(t)$ (resp. 
$\mathsf{ch}(t)$)  of a Coxeter transformation  (resp. of an affine Coxeter transformation) associated to $\GG$,

\begin{equation} \label{eq:Eb}  \mathsf{s}^0(t) =  \frac{\mathsf{ch}^{{}^{\boldsymbol{\circ}}}(t)}{\mathsf{ch}(t)}. \end{equation} 
\end{remark}
\medskip

Coxeter transformations (resp. affine Coxeter transformations) are  products  of reflections corresponding to the 
simple roots  (resp. affine simple roots).   There is one reflection in the product  for 
each node in the finite (resp. affine) Dynkin diagram.     There is a close
connection between  the spectrum of a Coxeter transformation and  the spectrum of the associated Cartan matrix  $\mathrm{C}$
which was described  in  \cite{BLM}  (see also \cite{C} for Dynkin
diagrams with odd cycles).    Eigenvalues  of $\mathrm{C}$ occur in pairs $\xi, 4-\xi$,   
and for each such a pair, $2-\xi$ and $\xi-2$ are eigenvalues of the 
adjacency matrix  of the Dynkin diagram.    
The results of \cite{BLM} and \cite{C} (see also \cite{Ste} and the discussion in \cite{D} for the finite diagrams)  imply that the 
eigenvalues of $\angstrom$  (resp. $\mathrm A$)  are
given by  $2\cos\left(\pi m/h\right)$ (resp.  $2\cos\big(\pi \widehat{m}/\what{h}\big)$)  where $m$ (resp. $\what m$)  ranges over
the exponents,  and $h$ (resp. $\what h$) is the Coxeter number (resp. affine Coxeter number).      In Table \ref{tab:expo-Cox},  
we display these exponents and numbers for the simply laced diagrams.    For multiply laced diagrams, they 
can be found in \cite[Table 1]{BLM}.

\begin{remark}  In the case not covered in Table \ref{tab:expo-Cox}, namely $\what {\mathsf{A}}_{2\ell}$,  there are $\ell$ conjugacy classes of Coxeter transformations
having  different spectra.  This case corresponds to the cyclic group $\Cs_{2\ell+1}$ of odd order,  which will be excluded in the next theorem.  
The characteristic polynomials  of the affine Coxeter transformations  for $\what {\mathsf{A}}_{2\ell}$ have been computed by Coleman \cite{C}.
The Poincar\'e polynomial $\mm^0(t)$ for all cyclic groups  is given  in Theorem \ref{T:cyc} of  Section \ref{S:cyc} below.  \end{remark} 

\begin{table}
$$\begin{tabular} {|c|c|c|} 
\hline
Dynkin &  Exponents   & Coxeter     \\
diagram & & number \\  \hline \hline
$\mathsf{A}_{n-1}$  & $1,2,\dots, n-1$ & \ $n$    \\ \hline
$\mathsf{D}_{n+2}$   & $1,3,\dots, 2n+1, n+1$ &\ $2n+2$    \\ \hline
$\mathsf{E}_6$  & 1, 4, 5, 7, 8, 11 & \ 12   \\  \hline
$\mathsf{E}_7$ & 1, 5, 7, 9, 11, 13, 17 & \ 18   \\  \hline
$\mathsf{E}_8$ & 1, 7, 11, 13, 17, 19, 23, 29 & \ 30   \\  \hline
$\what{\mathsf{A}}_{2\ell+1}$  & $0,1,1,\dots,\ell,\ell,\ell+1$ & \ $\ell+1$    \\  \hline
$\what{\mathsf{D}}_{2\ell+1}$   & $0,2,\dots,2\ell-2,2\ell-1,2\ell-1, 2\ell, \dots, 2(2\ell-1)$ &\ $2(2\ell-1)$    \\  \hline
$\what{\mathsf{D}}_{2\ell}$   & $0,1,\dots,\ell-1,\ell-1,\ell-1, \ell, \dots, 2\ell-2$ &\ $2\ell-2$    \\  \hline
$\what{\mathsf{E}}_6$  & 0, 2, 2, 3, 4, 4, 6 & \ 6   \\  \hline
$\what{\mathsf{E}}_7$ & 0, 3, 4, 6, 6, 8, 9,12 & \ 12 \\  \hline
$\what{\mathsf{E}}_8$ & 0, 6, 10, 12, 15, 18, 20, 24, 30 & \ 30   \\ 
\hline
\end{tabular}$$
\caption{Exponents and Coxeter numbers} \label{tab:expo-Cox}
\end{table} 
 \newpage

\begin{thm}\label{T:expo}  Let $\GG$ be a finite subgroup of $\SU_2$ such that $\GG \not \cong  \Cs_n$ for $n$ odd  and let  $\Vs = \CC^2$. 
Assume  $\mathrm A$ is the adjacency matrix of the representation graph $\mathcal R_{\Vs}(\GG)$ (the affine Dynkin diagram) and 
$\angstrom$ is  the adjacency matrix of the corresponding finite Dynkin diagram.   Let 
$\what{\Xi}$ (resp. $\Xi$ ) be the set of exponents 
and $\what {h}$  (resp. $h$) be the Coxeter number corresponding to the affine (resp. finite) Dynkin diagram.
Then  the Poincar\'e series for the $\GG$-invariants $\Ts(\Vs)^{\GG}$  in $\Ts(\Vs) = \bigoplus_{k \geq 0}\Vs^{\ot k}$ is
\begin{equation}\label{eq:cos}\mm^{0}(t) =  \frac{\det \left(\mathrm{I} - t  \angstrom \right)}{\det \left(\mathrm{I} - t \mathrm{A}\right)}
=  \frac{\det \left(\mathrm{I} - t  \angstrom \right)}{\prod_{g \in \Gamma} \left(1- \chi_\Vs(g)t\right)} =  \frac{\prod_{m \in \Xi} \left(1-2\cos\left(\frac{\pi m}{h}\right) t\right)}{\prod_{\what m \in \what{\Xi}} \left(1-2\cos\left(\frac{\pi \what m}{\what h}\right) t\right)}.\end{equation} \end{thm}

\begin{remark}  It is a consequence of \eqref{eq:cos} that the character values $\chi_\Vs(g)$ as $g$ ranges over a set $\Gamma$ of conjugacy class
representatives of $\GG$ are exactly the values  $2\cos\left(\pi \what m/ \what h\right)$, 
where $\what m$  is an  exponent in $\what \Xi$ and $\what h$ is the affine Coxeter number.  \end{remark}
 
 In subsequent sections,  we will derive other closed-form expressions for the Poincar\'e series $\mm^{0}(t)$ for 
all the finite subgroups of $\SU_2$.  In  the case of the
exceptional polyhedral groups $\GG$, we also consider  the Poincar\'e series $\mm^{\lam}(t)$ for all $\lambda \in \Lambda(\GG)$.   
Kor\'anyi \cite{Kor} has shown that the characteristic polynomial of the finite
Cartan matrices  of types $\mathsf{A,B,C,D}$ have expressions involving Chebyshev polynomials
of both the first and second kind.  In \cite{D},  Damianou has given an expression for the characteristic polynomial
of $\angstrom$ for all finite Dynkin diagrams equivalent to the one in the numerator  
of \eqref{eq:cos} and related these polynomials to Chebyshev polynomials. 
Since the closed-form expressions discussed here also will involve 
Chebyshev polynomials, we briefly review some facts about them. 

\begin{subsection}{Chebyshev polynomials} \end{subsection} 

The Chebyshev polynomials $\Tr_n(t)$  of the first kind are a set of orthogonal polynomials defined by the
recursion 
$$\Tr_0(t) = 1, \ \ \  \Tr_1(t) = t, \quad  \Tr_{n+1}(t) = 2t\Tr_n(t)  - \Tr_{n-1}(t)  \quad \hbox{\rm for all} \ \  n \geq 1.$$
Thus, the first few polynomials for $n \geq 2$ are 
\begin{align*} 
\Tr_2(t) &=	2t^2-1 \\	
\Tr_3(t)  & =	4t^3-3t \\	
\Tr_4(t) &=	8t^4-8t^2+1\\	
\Tr_5(t) &=	16t^5-20t^3+5t \\	
\Tr_6(t) &=	32t^6-48t^4+18t^2-1.\end{align*} 
The Chebyshev polynomials of the first kind play a critical role in approximating functions, where  the roots of the polynomials $\Tr_n(t)$, which can be expressed
in terms of cosines,  are used as nodes in polynomial interpolation.  The polynomials  have the following closed-form expressions:
\begin{align}\label{eq:T1} \Tr_n(t) & = \sum_{r = 0}^{\lfloor n/2\rfloor} {n  \choose {2r}}  t^{n-2r}(t^2-1)^r = t^n \sum_{r=0}^{\lfloor n/2\rfloor} {n \choose {2r}} (1-t^{-2})^r \\
& = 2^{n-1} \prod_{r=1}^n \left(t - \cos\left ( \textstyle{\frac{(2r-1)\pi}{2n}}\right)\right),  \label{eq:T2}
\end{align} which can be found in \cite{Ri} (see also \cite{D}).  

The Chebyshev polynomials $\Ur_n(t)$  of the second kind appear in the study of spherical harmonics in angular
momentum theory and in many other areas of mathematics and physics.  They have a similar recursive definition, 
 $$\Ur_0(t) = 1, \ \ \  \Ur_1(t) =2t, \quad  \Ur_{n+1}(t) = 2t\Ur_n(t)  - \Ur_{n-1}(t)  \quad \hbox{\rm for all} \ \  n \geq 1.$$
 A slight variation of these polynomials,  which arises frequently and will be useful in what follows, are the polynomials
 $\pr_n(t)  =  \Ur_n(t/2)$,   which satisfy the relations
\begin{equation}\label{eq:pr} \pr_0(t) = 1, \ \ \  \pr_1(t) =t, \quad  \pr_{n+1}(t) = t\pr_n(t)  - \pr_{n-1}(t)  \quad \hbox{\rm for all} \ \  n \geq 1. \end{equation} 
The first few polynomials in these series for $n \geq 2$ are  
$$\begin{tabular} {cccc}
$\Ur_2(t) =$ & $4t^2-1$ \hspace{2.2cm} & \hspace{1.2cm}$ \pr_2(t) =$ &  \hspace{-1.8cm}$t^2-1$  \\	
$\Ur_3(t) =$ &  $8t^3-4t$ \hspace{2cm} & \hspace{1.2cm} $\pr_3(t)  =$ & \hspace{-1.85cm} $t^3-2t$  \\	 
$\Ur_4(t) =$ & $16t^4-12t^2+1$\hspace{1.1cm} & \hspace{1.2cm}$ \pr_4(t) =$  & \hspace{-1.1cm} $t^4-3t^2+1$\\	
$\Ur_5(t) =$ & $32t^5-32t^3+6t$\hspace{1cm} & \hspace{1.2cm} $\pr_5(t) =$ &\hspace{-1.1cm} $t^5-4t^3+3t$ \\	
$\Ur_6(t) =$ & $64t^6-80t^4+24t^2-1$ & \hspace{1.2cm} $\pr_6(t)=$&$t^6 -5t^4 +6t^2 -1$.\end{tabular} $$ 

The Chebyshev polynomial $\Ur_n(t)$ has simple roots given by $\cos\big(\pi r/(n+1)\big)$
where $r = 1,\dots, n$.      Thus, the roots of  $\pr_n(t)$ are $2\cos\big(\pi r/(n+1)\big)$ for $r = 1,\dots, n$.  
This leads to the explicit expressions for the polynomials $\Ur_n(t)$ given in \cite{Ri} (see also \cite[Sec.~6.3]{D}),

\begin{align} \label{eq:U1} \Ur_n(t) &= \sum_{r = 0}^{\lfloor n/2\rfloor}  (-1)^r {n-r  \choose r}  (2t)^{n-2r} \ = \  2^n\prod_{r=1}^n  \left(t - \cos\left(\frac{\pi r}{n+1}\right)\right) \\
\pr_n(t) &= \sum_{r = 0}^{\lfloor n/2\rfloor}  (-1)^r {n-r  \choose r} \, t^{n-2r} 
\ = \  \prod_{r=1}^n  \left(t - 2\cos\left(\frac{\pi r}{n+1}\right)\right).  \label{eq:U2} \end{align} 

There are many identities relating the Chebyshev polynomials of the second kind to those of the first kind.  A particularly useful one for our purposes is 
\begin{equation} \label{eq:UT} \Ur_n(t) - \Ur_{n-2}(t) = 2 \mathrm{T}_n(t). \end{equation} 
 
\begin{subsection}{Cyclic groups} \label{S:cyc} \end{subsection}
Assume for $n \geq 3$  that 
\begin{equation}\label{eq:Cycgen} \mathrm{z}= \left (\begin{array}{cc} \zeta_n^{-1}  & 0 \\ 0 & \zeta_n \end{array}\right)\end{equation} where $\zeta_n = e^{2\pi i/n}$, a primitive $n$th root of unity in $\CC$, and let 
$\Cs_n$ be the cyclic subgroup of $\mathsf{SU}_2$ generated by  $\mathrm{z}$.    
The irreducible  modules for  $ \Cs_n$ are all one-dimensional and are given by  $\Cs_n^{(\ell)}= \CC\vs_\ell$ for $\ell = 0,1, \dots, n-1$,  where
$\mathrm{z} \vs_\ell = \zeta^\ell \vs_\ell$.  Thus,   $\Lambda(\Cs_n) = \{0,1,\dots,n-1\}$,  and $\Cs_n^{(j)} \cong \Cs_n^{(\ell)}$  whenever   $j \equiv \ell \modd n$. 
The  natural $\Cs_n$-module $\Vs$  
 of $2 \times 1$ column vectors, which $\Cs_n$ acts on by matrix
 multiplication, can be identified with the module  $\Cs_n^{(-1)}  \oplus \Cs_n^{(1)} = \Cs_n^{(n-1)}  \oplus \Cs_n^{(1)}$.   As $\GG$ is abelian, the conjugacy class representatives are simply all the elements $\mathrm{z}^r$, $r=0,1,\dots,n-1$,  of $\GG$.       The character value for $\mathrm{z}^r$ on $\Vs$ is  $\chi_{\Vs}(\mathrm{z}^r) = \zeta_n^{-r} + \zeta_n^r = 2 \cos(2\pi r/n)$.   Thus,  \eqref{eq:inv} becomes in this case
\begin{equation}\label{eq:invcyc}\mm^{0}(t) =  \frac{\det \left(\mathrm{I} - t  \angstrom \right)}{\det \left(\mathrm{I} - t \mathrm{A}\right)}
=  \frac{\det \left(\mathrm{I} - t  \angstrom \right)}{\prod_{r=0}^{n-1} \left(1- 2 \cos\left(\textstyle{\frac{2\pi r}{n}}\right)t\right)}.
\end{equation}

The matrix  $\angstrom$ in this equation is the adjacency matrix of the finite Dynkin diagram of type 
$\mathsf{A}_{n-1}$ obtained from the Dynkin diagram $\what{\mathsf{A}}_{n-1}$  in \eqref{eq:dd} by removing the affine node.  Thus,
$\angstrom$ is the tridiagonal matrix, 
$$\angstrom = \left ( \begin{matrix}  0 & 1  & 0 &  \cdots & 0 \\
1 & 0 & 1 & \dots  & 0 \\
\vdots &  \ddots & \ddots & \ddots & 0 \\
0 & \ldots & 1 & 0  & 1 \\
0 & 0 & \ldots &   1 & 0 \end{matrix} \right ).$$

Let $\arm_{n-1}(t)$
be the determinant $\det \left(\mathrm{I} - t  \angstrom \right)$ in the $\mathsf{A}_{n-1}$-case,
and set $\arm_0(t) = 1 = \arm_1(t)$.    It is easy to see using cofactor expansion on $\mathrm{I} - t  \angstrom$ that the following recursion relation holds,
\begin{equation} \arm_{n+1}(t) =  \arm_n(t) - t^2 \arm_{n-1}(t)  \qquad  \hbox{\rm for} \ \ n \geq 1. \end{equation} 
An easy inductive argument using the recursion relation for the polynomials $\pr_n(t)$ in \eqref{eq:pr}  shows that
\begin{equation}\label{eq:arm} \arm_n(t) = t^n \pr_n(t^{-1}) \qquad \hbox{\rm for all} \ \  n \geq 0. \end{equation} 

Since the roots of  $\pr_n(t)$ are $2\cos(\pi r/ (n+1))$ for $r = 1,\dots, n$,  (compare \eqref{eq:U2}),   we have 
\begin{align}\begin{split} \arm_n(t) & = t^{n}\pr_n(t^{-1}) = t^n \prod_{r=1}^n \left(t^{-1} - 2\cos\left(\frac{\pi r}{n+1}\right)\right) \\
& =  \prod_{r=1}^n \left(1 - 2\cos\left(\frac{\pi r}{n+1}\right)\, t\right).\end{split}\end{align}
Then \eqref{eq:U2} implies that  $\arm_n(t)$ has the closed-form expression, 

\begin{align}\label{eq:armexp} \arm_n(t)  = t^{n}\pr_n(t^{-1}) = t^n \sum_{r=0}^{\lfloor n/2\rfloor}  (-1)^r {n-r  \choose r}  (t^{-1})^{n-2r}  
= \sum_{r=0}^{\lfloor n/2\rfloor}  (-1)^r {n-r  \choose r}  t^{2r}. \end{align}
(Compare the expressions in  \cite[Sec.~6.3]{D}.)

The denominator in \eqref{eq:invcyc} is also related to Chebyshev polynomials, as cofactor expansion on $\mathrm{I} - t \mathrm{A}$ 
combined with \eqref{eq:arm} shows that
\begin{align}\begin{split}  
\det \left(\mathrm{I} - t \mathrm{A}\right) &= \arm_{n-1}(t) - 2 t^2 \arm_{n-2}(t) - 2t^n \\
&=t^{n-1} \pr_{n-1}(t^{-1}) - 2 t^n \pr_{n-2}(t^{-1}) - 2t^n \\
&=t^{n}\left( t^{-1} \pr_{n-1}(t^{-1}) - \pr_{n-2}(t^{-1})\right) - t^n \pr_{n-2}(t^{-1}) - 2t^n \\
&=t^n \left( \pr_n(t^{-1}) - \pr_{n-2}(t^{-1}) - 2 \right). \end{split} \end{align}
Then  \eqref{eq:UT} and \eqref{eq:T1} imply
$$\Ur_n(t) - \Ur_{n-2}(t) = 2 \mathrm{T}_n(t) = 2 \sum_{r=0}^{\lfloor n/2 \rfloor} {n \choose 2r} (t^2 -1)^r t^{n-2r},$$
so  that   
\begin{align}\begin{split}  
\det \left(\mathrm{I} - t \mathrm{A}\right) &= t^n \left( \pr_n(t^{-1}) - \pr_{n-2}(t^{-1}) - 2 \right) \\
&=  2 t^n \, \sum_{r=0}^{\lfloor n/2 \rfloor} {n \choose 2r} \left(\frac{t^{-2}}{4}- 1\right)^r  \left(\frac{t^{-1}}{2}\right)^{n-2r} \ - \ 2t^n \\
& = 2^{1-n}\sum_{r=0}^{\lfloor n/2 \rfloor} {n \choose 2r}(1-4t^2)^r   \ - \ 2t^n.   \end{split} \end{align} 

We summarize what we have shown for the cyclic case in the next theorem. 
\medskip

\begin{thm} \label{T:cyc}  Assume $\GG$ is the cyclic group $\Cs_n$ and let $\Vs = \CC^2 = \Cs_n^{(-1)} \oplus \Cs_n^{(1)}$. 
Then the Poincar\'e series $\mm^{0}(t)$ for the $\GG$-invariants $\Ts(\Vs)^\GG$  in $\Ts(\Vs) = \bigoplus_{k \geq 0} \Vs^{\ot k}$ is given by

\begin{align}\begin{split} \label{eq:cycinv2} \mm^{0}(t)  &=   \frac{{\displaystyle{\prod_{r=1}^{n-1}} \Big(1- 2 \cos(\textstyle{\frac{\pi r}{n}})t\Big)}}{\displaystyle{\prod_{r=0}^{n-1}} \Big(1- 2 \cos(\textstyle{\frac{2\pi r}{n}})t\Big)}  
= \frac{ \displaystyle{ \sum_{r=0}^{\lfloor (n-1)/2\rfloor}  (-1)^r {n-1-r  \choose r}  t^{2r} }}
{\displaystyle{2^{1-n}\sum_{r=0}^{\lfloor n/2 \rfloor} {n \choose 2r}(1-4t^2)^r   \ - \ 2t^n}}. \end{split}\end{align} 
\end{thm}

\begin{remark}  When $\GG = \Cs_n$ with $n = 2(\ell+1)$ and  $\ell \geq 1$,  the cosine expression for
$\mm^0(t)$  in \eqref{eq:cycinv2}  gives the same result as
Theorem \ref{T:expo}.    Indeed,  $n-1 = 2\ell+1$, \   $\cos\left(\pi m/(\ell+1)\right) = \cos\left(2\pi m/n\right)$,  and the product 
of the factors \  $\left(1-2\cos\left(\pi m\ (\ell+1) \right) t\right)$ \  as $m$ ranges over the elements of \  $\what \Xi \, = \, \{0, 1, 1, \dots,\, \ell, \break  \ell,\, \ell+1\}$ is the
same as the product of the terms  $\left(1-2\cos\left(2\pi r/ n\right)t\right)$ for $r=0,1,\dots, n-1$,  since
$\cos\left(2\pi (n-r)/n\right) = \cos\left(2\pi r/n\right)$ for all $r= 0,1,\dots, \ell$.  
\end{remark} 

In Table \ref{tabl:cyc},  the numerator and denominator polynomials in \eqref{eq:cycinv2} and the Poincar\'e series 
$\mm^0(t)$ are
displayed  for $\Cs_n$,  $n = 3,4,5,6,7$.  
\begin{table}[h]
$$\begin{tabular} {|c|| c|c|c|} 
\hline
 & $\det(\mathrm{I}-t \angstrom)$ & $\det(\mathrm{I}-t \mathrm{A})$    & $\begin{matrix} \mm^{0}(t) =  \\
\det \left(\mathrm{I} - t  \angstrom \right)/\det \left(\mathrm{I} - t \mathrm{A}\right) \end{matrix}$  \\ \hline
\hline
$(\Cs_3, \what{\mathsf{A}}_2)$ & $1-t^2$ & $1-3t^2-2t^3$ & $1+2t^2 + 2t^3 + 6t^4$     \\ 
  & & & $+ 10t^5 + 22t^6 + \, \cdots \,$ \\ \hline
$(\Cs_4, \what{\mathsf{A}}_3)$  & $1-2t^2$ & $1-4t^2$ &$1 + 2t^2 + 8t^4$ \\
 & & & $32t^6 + 128 t^7 \, + \, \cdots$     \\ \hline
$(\Cs_5, \what{\mathsf{A}}_4)$  & $1-3t^2+t^4$ & $1-5t^2+5t^4-2t^5$ & $1+2t^2+6t^4 + 2t^5$\\
& & & $+20t^6 + 14t^7 \, + \, \cdots$    \\  \hline
$(\Cs_6, \what{\mathsf{A}}_5)$  & $1-4t^2+3t^4$ & $1-6t^2+9t^4-4t^6$ & $1+2t^2 + 6t^4$ \\
& & & $22t^6 +86 t^8\, + \, \cdots$     \\ \hline
$(\Cs_7, \what{\mathsf{A}}_6)$  & $1-5t^2$ & $1-7t^2+14t^4$ & $1+2t^2+6t^4$    \\
& $+6t^4-t^6$ & $-7t^6-2t^7$ & $+20t^6 +2t^7 + 70 t^8 \,+\, \cdots$  \\
\hline
\end{tabular}$$   \caption{Poincar\'e series $\mm^{0}(t)$ for the cyclic groups $\Cs_n$, \  $3\leq n \leq 7$ \label{tabl:cyc} }
 \end{table}  
   
 The Bratteli diagram for $\Cs_n$ and $\Vs = \CC^2$  is Pascal's triangle on a cylinder of ``diameter'' $\tilde n$,
 where $\tilde n = n$ if $n$ is odd and $\tilde n = \half n$ if $n$ is even.   Pictured in  Figure \ref{fig:Cyc5}  is the Bratteli diagram for $\Cs_5$.    The subscripts on the white (trivial) node  correspond to the Poincar\'e series $1+2t^2+6t^4 + 2t^5+20t^6 + 14t^7+\cdots$ in the third line of  Table \ref{tabl:cyc}. 
 \begin{figure}[ht!] 
$$
\begin{array}{c} \includegraphics[scale=.7,page=6]{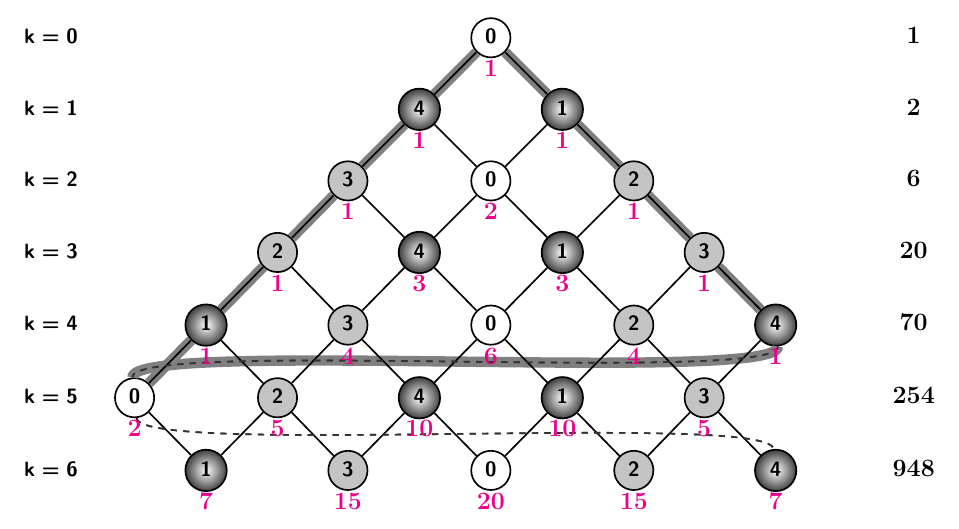} \end{array}
$$ 
\caption{Levels $k=0,1,\dots,6$ of the Bratteli diagram  $\mathcal{B}_\Vs(\Cs_5)$ for $\Vs = \CC^2$}\label{fig:Cyc5}
\end{figure}  

\begin{remark}\label{R:JBC} It was shown in \cite[Sec.~1.6]{BBH}  by using  the basic construction of Jones \cite{J1} that  for all finite subgroups $\GG$ of $\SU_2$, the edges in $\mathcal{B}_\Vs(\GG)$ between level $k$ and level $k+1$ 
 that are NOT obtained from edges between level $k-1$ and level $k$ by reflection over level $k$, exactly form the 
 representation graph (affine Dynkin diagram),  except in the $\Cs_n$ case for $n$ odd where it is the double of the Dynkin diagram.   This is indicated by the shaded edges in Figures \ref{fig:Cyc5}, \ref{fig:D6}, 5, 6, and 7 below and has  important implications for the structure of the centralizer algebras $\Zs_k(\GG)$. \end{remark}     
 
\begin{subsection}{Binary dihedral groups} \end{subsection}
Assume for $n \geq 2$,  that 
\begin{equation}\label{eq:Cycgen} \mathrm{x} = \left (\begin{array}{cc} \zeta_{2n}^{-1}  & 0 \\ 0 & \zeta_{2n} \end{array}\right ) 
\quad \hbox{\rm and} \quad \mathrm{y} = \left (\begin{array}{cc} 0  & i \\ i & 0\end{array}\right ), \end{equation}
where $ i = \sqrt{-1}$ and $\zeta_{2n} = e^{2\pi i/ 2n} = e^{\pi i/ n}$, a primitive $2n$th root of unity.  
The elements $\mathrm{x}$ and $\mathrm{y}$ generate a binary dihedral subgroup  $\DD_n$  of order $4n$ in $\SU_2$.   
There are four irreducible $\DD_n$-modules of dimension $1$, and $n-1$ irreducible modules of dimension $2$.   
The defining module $\Vs = \CC^2$ is irreducible as a $\DD_n$-module.   

We take as the conjugacy class representatives of $\DD_n$ the elements 
in  $\Gamma = \{\pm \mathrm{I},  \ \mathrm{x}^r, (r = 1,\dots, n-1),  \mathrm{y}, \ 
\mathrm{yx}\}$ .    Then computing their traces gives

\begin{align}\label{eq:Ddet} \det \left(\mathrm{I} - t \mathrm{A}\right)  & =  \prod_{g \in \Gamma} \left(1- \chi_\Vs(g)t \right ) =
(1-2t)(1+2t)\prod_{r=1}^{n-1} \left(1- 2\cos\left(\frac{\pi r}{n}\right) t \right)\nonumber  \\    &=  (1-4t^2) \mathrm{a}_{n-1}(t) = 
(1-4t^2)  \sum_{r=0}^{\lfloor (n-1)/2\rfloor}  (-1)^r {n-1-r  \choose r}  t^{2r}, \end{align} 
where the expression for $\arm_{n-1}(t)$ follows from \eqref{eq:armexp}.   

The matrix  $\angstrom$  is the adjacency matrix of the finite Dynkin diagram of type 
$\mathsf{D}_{n+2}$ obtained from the Dynkin diagram $\what{\mathsf{D}}_{n+2}$  in \eqref{eq:dd}.   
Let $\dr_{n}(t)$
be the determinant $\det \left(\mathrm{I} - t  \angstrom \right)$ in the $\mathsf{D}_{n+2}$-case,
and set $\dr_0(t) = 1,  \ \dr_1(t) = 1- 2t^2$.    It is easy to see using cofactor expansion on the
matrix $\mathrm{I} - t  \angstrom$  that the following recursion relation holds,
$$\dr_{n+1}(t) =  \dr_n(t) - t^2 \dr_{n-1}(t)  \qquad  \hbox{\rm for} \ \ n \geq 1.$$

The polynomials 
$2 t^{n+1} \Tr_{n+1}(t^{-1}/2)$, where $\Tr_{n+1}(t)$ is the Chebyshev polynomial of the first kind,  satisfy the same initial conditions; namely,  
$$2 t^{n+1} \Tr_{n+1}\left(\frac{t^{-1}}{2}\right) = \begin{cases} 1 & \qquad \hbox{\rm for} \ \   n = 0 \\  1-2t^2   & \qquad \hbox{\rm for} \ \   n =1,
\end{cases} $$
 and the same recursion relation as the    
polynomials $\dr_n(t)$.   As a consequence,  we can conclude
$$\dr_n(t)  = 2 t^{n+1} \Tr_{n+1}\left(\frac{t^{-1}}{2}\right).$$
Combining that relation with the identities in \eqref{eq:T1} and \eqref{eq:T2} gives 

\begin{align} \dr_n(t) &= 2t^{n+1}\left(\frac{t^{-1}}{2}\right)^{n+1} \sum_{r=0}^{\lfloor {n+1}/2\rfloor} {n+1  \choose {2r}} \left (1-\left(\frac{t^{-1}}{2}\right)^{-2}\right)^{r} \nonumber  \\
& = 2^{-n}\sum_{r=0}^{\lfloor (n+1)/2\rfloor} {n+1 \choose {2r}}(1-4t^2)^r \label{eq:d1} \\ &=  \prod_{r=1}^{n+1} \left(1 - 2\cos\left ( \textstyle{\frac{(2r-1)\pi}{2(n+1)}}\right)t \right). \label{eq:d2}  \end{align}
(Compare with  \cite[Sec.~6.5]{D}, which computes the characteristic
polynomial of $\angstrom$ for the $\mathsf{D}$-case.)   The expressions in \eqref{eq:d1} and \eqref{eq:d2}  together  with \eqref{eq:Ddet} imply  the next result.  
\bigskip

\begin{thm} \label{T:dih}  Assume $\GG$ is the binary dihedral group $\DD_n$,  and let $\Vs = \CC^2$. 
Then the Poincar\'e series $\mm^{0}(t)$ for the $\GG$-invariants in $\Ts(\Vs) = \bigoplus_{k \geq 0} \Vs^{\ot k}$ is given by

\begin{align} \label{eq:dihedinv} \mm^{0}(t)  &=  \frac{ 2 t^{n+1} \Tr_{n+1}\left(\frac{t^{-1}}{2}\right)}{(1-4t^2)\displaystyle{\prod_{r=1}^{n-1}} \Big(1- 2 \cos(\textstyle{\frac{\pi r}{n}})t\Big)} 
=  \frac{\displaystyle{\prod_{r=1}^{n+1}} \left(1-2\cos\left(\textstyle{\frac{(2r-1)\pi}{2(n+1)}}\right)t \right)}  {(1-4t^2)\displaystyle{\prod_{r=1}^{n-1}} \Big(1- 2 \cos(\textstyle{\frac{\pi r}{n}})t\Big)}  \\
&=   \frac{\displaystyle{2^{-n}\sum_{r=0}^{\lfloor (n+1)/2\rfloor} {n+1 \choose {2r}}(1-4t^2)^r}}{(1-4t^2)\displaystyle{ \sum_{r=0}^{\lfloor (n-1)/2\rfloor}  (-1)^r {n-1-r  \choose r}  t^{2r} }}. \nonumber
 \end{align} 
\end{thm}

Table \ref{tab:dih} below displays the polynomials $\det(\mathrm{I}-t \angstrom)$ and  $\det(\mathrm{I}-t \mathrm{A})$ and Poincar\'e series 
$\mm^{0}(t)=  \det \left(\mathrm{I} - t  \angstrom\right) / \det \left(\mathrm{I} - t \mathrm{A}\right)$  for $\DD_n$,  $n = 2,3,4,5,6$.  
The Bratteli diagram for the binary dihedral group $\DD_6$ and $\Vs = \CC^2$  is  pictured in  Figure \ref{fig:D6}.   
The shaded edges give the 
 representation graph (affine Dynkin diagram $\what{\mathsf{D}}_8$). 
 The subscripts on the white node, which corresponds to the trivial $\DD_6$-module,  are the coefficients of the
Poincar\'e series $\mm^{0}(t)$ in the last line of Table \ref{tab:dih}.  
\bigskip  \medskip
 
 \begin{table}[h] \begin{center}{
\begin{tabular} {|c|| c|c|c|} 
\hline
 & $\det(\mathrm{I}-t \angstrom)$ & $\det(\mathrm{I}-t \mathrm{A})$    & $\mm^{0}(t)$    \\ \hline
\hline
$(\DD_2, \what{\mathsf{D}}_4)$ & $1-3t^2$ & $1-4t^2$ & $1+t^2 + 4t^4 + 16t^6$     \\ 
  & & & $+ 64t^8 + 256t^{10} + \, \cdots \,$ \\ \hline
$(\DD_3, \what{\mathsf{D}}_5)$  & $1-4t^2+2t^4$ & $1-5t^2+4t^4$ & $1+t^2 + 3t^4 + 11t^6$  \\
 & & & $43t^8 + 171 t^{10} \, + \, \cdots$     \\ \hline
$(\DD_4, \what{\mathsf{D}}_6)$  & $1-5t^2+5t^4$ & $1-6t^2+8t^4$ & $1+t^2+3t^4+10t^6$ \\
& & & $+36t^8 + 136t^{10} \, + \, \cdots$    \\  \hline
$(\DD_5, \what{\mathsf{D}}_7)$  & $1-6t^2$\qquad & $1-7t^2$\quad & $1+t^2+3t^4+10t^6$ \\
&$+9t^4-2t^6$ & $+13t^4-4t^6$  & $35t^8 +118 t^{10}\, + \, \cdots$     \\  
 \hline
$(\DD_6, \what{\mathsf{D}}_8)$  & $1-7t^2$\qquad & $1-8t^2$\quad & $1+t^2+3t^4+10t^6$    \\
&$+14t^4-7t^6$&$+19t^4-12t^6$& $+35 t^8 + 126t^{10} \,+\, \cdots$  \\
\hline
\end{tabular} }\end{center} 
 \caption{Poincar\'e series $\mm^{0}(t)$ for the 
binary  dihedral groups $\DD_n$, \  $2\leq n \leq 6$ \label{tab:dih} }
 \end{table}  
 
 \begin{figure}[ht!] 
$$
\begin{array}{c} \includegraphics[scale=.75,page=7]{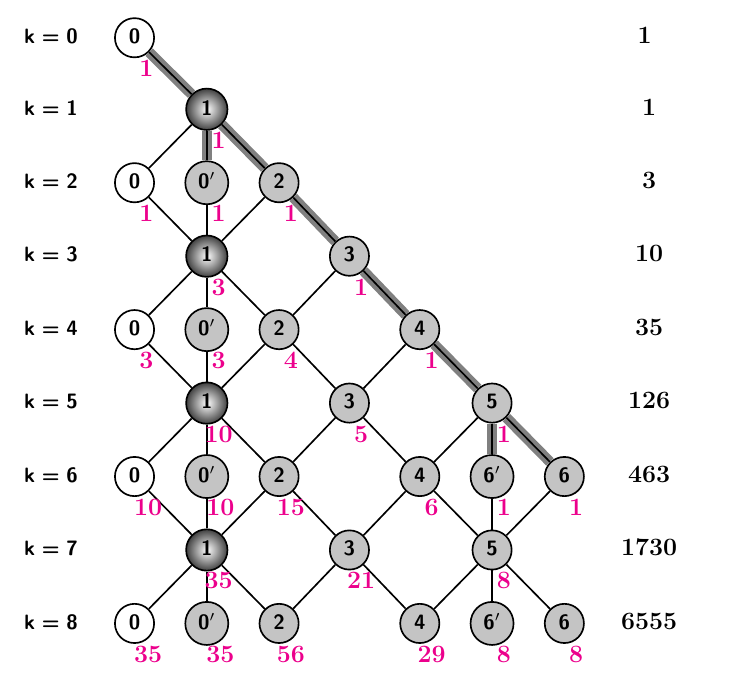} \end{array}
$$ \vspace{-.8cm} 
 \caption{Levels $k=0,1,\dots, 8$ of the Bratteli diagram  $\mathcal{B}_\Vs(\DD_6)$ for $\Vs = \CC^2$}\label{fig:D6}
\end{figure}

\newpage \hspace{.1cm}

\begin{subsection}{Exceptional binary polyhedral groups} \end{subsection}
In this final section, we present analogous results for the subgroups $\TT, \OO,$ and $\II$ of $\SU_2$.  
The denominator $\det(\mathrm{I}-t \mathrm{A})$ in the Poincar\'e series  can be computed by applying Theorem \ref{T:expo} with the exponents
in Table \ref{tab:expo-Cox}.  Alternatively, one can use 
the fact that $\det(\mathrm{I}-t \mathrm{A}) = \prod_{g \in \Gamma} (1-\chi_{\Vs}(g) t)$ and read off the character
values,  for example,  from  \cite[Tables A.9, A.12, and A.19]{Stk}.   The determinants  $\det(\mathrm{I}-t \mathrm{A})$
and $\det(\mathrm{I}-t \angstrom)$  also can be computed by
hand or by using a convenient software package.   In applying Theorem \ref{T:expo} to evaluate  
$\det(\mathrm{I}-t \angstrom)$, it is helpful to use the fact that the exponents of the finite
Dynkin diagram occur in pairs $m, m'$ such that $m+m' = h$ and $\cos\left(\pi m/h\right) = -\cos\left(\pi m'/h\right)$
to get the results below.    In Table \ref{tabu:Ex},  $\phi = \half(1 + \sqrt{5})$ (the golden ratio),  and $\phi^* = \half(1 - \sqrt{5})$. 
The subscripts on the white (trivial) nodes in the Bratteli diagrams $\mathcal B_{\Vs}(\GG)$ below  for $\GG = \TT,\OO,\II$  correspond to the
coefficients of  Poincar\'e series $\mm^0(t)$ in Table \ref{tabu:Ex}. \newpage

\vspace{-1cm}
\begin{table}[h]
$$\begin{tabular} {|c||c|} 
\hline \hline  
 & $(\TT, \what{\mathsf{E}}_6)$    \\  \hline \hline 
$\det(\mathrm{I}-t \angstrom)$  &  $\prod_{m=1,4,5}  \left(1-4\cos^2\left(\frac{m\pi}{12}\right)t^2\right) = 1-5t^2 +5t^4 -t^6$   \\ \hline
$\det(\mathrm{I}-t \mathrm{A})$ & $(1-2t)(1+2t) (1-t)^2(1+t)^2 = 1-6t^2+9t^4-4t^6$  \\ \hline 
$\mm^0(t)$ & $1+t^2+2t^4+6t^6+22t^8 + 86t^{10} \,+\, \cdots$  \\ 
\hline \hline
& $(\OO, \what{\mathsf{E}}_7)$   \\ \hline \hline
$\det(\mathrm{I}-t \angstrom)$  & $\prod_{m=1,5,7} \left(1-4\cos^2\left(\frac{m\pi}{18}\right)t^2\right) = 1-6t^2 +9t^4 -3t^6$ \\
\hline
$\det(\mathrm{I}-t \mathrm{A})$  &$(1-2t)(1+2t) (1-t)(1+t)(1+\sqrt{2}t)(1-\sqrt{2}t)$
\\ & \qquad \qquad  $=  1-7t^2+14t^4-8t^6$  \\   \hline
$\mm^0(t)$ & $1+t^2+2t^4+5t^6+15t^8 + 51t^{10} \,+\, \cdots$ \\
\hline \hline
& $(\II, \what{\mathsf{E}}_8)$  \\  \hline  \hline
$\det(\mathrm{I}-t \angstrom)$  & $\prod_{m=1,7,11,13} \left(1-4\cos^2\left(\frac{m\pi}{30}\right)t^2\right) = 1-7t^2 +14t^4 -8t^6 +t^8$ \\
\hline
$\det(\mathrm{I}-t \mathrm{A})$ &$(1-2t)(1+2t)(1-t)(1+t)$\qquad \qquad \qquad \qquad \hfil\\
& \qquad \qquad\qquad $\times (1-\phi t)(1-\phi^* t)(1+\phi t)(1+ \phi^* t)$  \\
& $\ \qquad \qquad \quad  = (1-4t^2)(1-t^2)(1-t-t^2)(1+t-t^2)$ \\ 
& $= 1-8t^2 +20t^4-17t^6 + 4t^8 $ \\ \hline 
$\mm^0(t)$ & $ 1+t^2 +2t^4 +5t^6 +14t^8+42t^{10} + 133 t^{12} \,+\, \cdots$ \\  \hline  
\end{tabular} $$ \vspace{-1cm}
\begin{center}{ \caption{Poincar\'e series $\mm^0(t) = \det(\mathrm{I}-t\angstrom) / \det(\mathrm{I}-t \mathrm{A})$ for the exceptional polyhedral groups }\label{tabu:Ex}}\end{center} 
\end{table}   
\vspace{-1cm}

\begin{center}{\includegraphics[scale=.7,page=8]{poincare-diagrams.pdf}}\end{center}
\vspace{-.8cm}
\begin{center}{Figure 5: Levels $k=0,1,\dots,8$ of the Bratteli diagram  $\mathcal{B}_\Vs(\TT)$  for $\Vs = \CC^2$}\label{fig:TT} \end{center}

\begin{center}{ \includegraphics[scale=.78,page=9]{poincare-diagrams.pdf}}\end{center}
\vspace{-.7cm}
 \begin{center}{Figure 6: Levels $k=0,1,\dots,8$ of the Bratteli diagram  $\mathcal{B}_\Vs(\OO)$ for $\Vs = \CC^2$}\label{fig:OO}\end{center}  

\begin{center}{ \includegraphics[scale=.75,page=10]{poincare-diagrams.pdf}} \end{center}   
\vspace{-.5cm}
 \begin{center}{Figure 7: Levels $k=0,1,\dots,8$ of the Bratteli diagram  $\mathcal{B}_\Vs(\II)$ for $\Vs = \CC^2$} \label{fig:II} \end{center} 

Recall from \eqref{eq:main} that the Poincar\'e series $ \mm^\mu(t)$  for the multiplicity of the irreducible $\GG$-module $\GG^\mu$
 in the tensor algebra $\Ts(\Vs)$  is given by  $ \mm^\mu(t) =$  \break  $\det(\mathrm{M}^\mu)/\det(\mathrm{I}-t\mathrm{A}) =
 \det(\mathrm{M}^\mu)/\prod_{g \in \Gamma} \left(1- \chi_\Vs(g)t\right)$, where $\mathrm{M}^\mu$ is the matrix obtained
 from $\mathrm{I}-t\mathrm{A}$ by replacing the column indexed by $\mu$ by the column $\underline{\delta} = \left (\begin{smallmatrix}  1\\ 0 \\ \vdots \\  \\ 0 \end{smallmatrix}\right )$ in Theorem \ref{T:main}.
 The series $\mm^\mu(t)$ can also be computed using Remark \ref{R:mmu} since $\mm^0(t)$ is known
from  Table \ref{tabu:Ex}.   In the diagrams below,  we attach to each
	 node $\mu$ of the affine Dynkin diagrams $\what {\mathsf{E}}_6, \what {\mathsf{E}}_7, \what {\mathsf{E}}_8$,   the  polynomial $\det(\mathrm{M}^\mu)$.  
 The polynomial for $\mu=0$ is the same as $\det(\mathrm{I}-t\angstrom)$ in the previous table.   
\bigskip
 
 \begin{center}{\hspace{-.25cm} \includegraphics[scale=.95, page=11]{poincare-diagrams.pdf}
\hspace{-.75cm} \includegraphics[scale=.95,page=12]{poincare-diagrams.pdf} }\end{center} 
\vspace{-.45cm} 
\begin{center}{Figure 8: \  $\det(\mathrm{M}^\mu)$ for the binary tetrahedral and octahedral groups $\TT$ and $\OO$ \label{fig:TO}}
 \end{center}
 \vspace{-1cm} 
\begin{center}{\includegraphics[scale=1,page=13]{poincare-diagrams.pdf}}\end{center}

\vspace{-.8cm}
\begin{center}{Figure 9: \  $\det(\mathrm{M}^\mu)$ for the binary icosahedral group $\II$ \label{fig:II}}  
\end{center} 

\begin{remark} An inductive argument on the edge connections in the Bratteli diagram can be used to 
determine formulas for the irreducible $\mathsf{Z}_k(\GG)$-modules (i.e. for
the multiplicities $\mm_k^\lambda$)  for the exceptional polyhedral groups.  The results of
applying this method were given in \cite{BBH}. 
 For example,  when $\lam = 0$, then $\mm_k^0 = 0$ unless $k$ is even, and  for $k = 2n\geq 2$, we have from \cite[Sec.~4.3]{BBH} that  
\begin{equation*}\mm_k^0 = \begin{cases} \frac{1}{12}(4^n+8) & \qquad \hbox{\rm if} \ \ \GG = \TT, \\
\frac{1}{24}(4^{n} +6\cdot 2^n+8)  & \qquad \hbox{\rm if} \ \ \GG = \OO,\\
\frac{1}{60}(4^n + 12 \mathsf{L}_{2n} + 20) & \qquad \hbox{\rm if} \ \ \GG = \II, 
\end{cases}\end{equation*} 
where $L_{2n}$ is the  (2n)th Lucas number.   The Lucas numbers  $L_r$, $r \geq 0$,  
satisfy the Fibonacci recursion $L_{r+1} = L_{r} + L_{r-1}$  but
start from the initial conditions  $L_0 = 2, \ L_1 = 1$.   
Similar expressions for $\mm_k^\lambda$ for all $\lambda \in \Lambda(\II)$ also involve Lucas numbers.   The expressions for $\mm_k^\lambda$ in \cite[Sec.~4.3]{BBH}
can be put into a generating series,  which can also be used to determine  the Poincar\'e series  $\mm_k^\lam(t)$ for the exceptional groups. \end{remark}

\medskip 

\noindent \textit{\small Department of Mathematics, University of Wisconsin-Madison, Madison, WI 53706, USA}\\
{\small benkart@math.wisc.edu}

 \end{document}